\documentclass
{amsart}
\usepackage{graphicx}
\newtheorem{thm}{Theorem}[section]
\theoremstyle{definition}
\newtheorem{cor}[thm]{Corollary}
\newtheorem{prop}[thm]{Proposition}
\newtheorem{defn}[thm]{Definition}

\newtheorem{rem}[thm]{Remark}

\numberwithin{equation}{section}
\begin{document}
\title[Strongly $\psi $-$2$-absorbing second submodules]{Strongly $\psi $-$2$-absorbing second submodules}

\author%
{H. Ansari-Toroghy**}

\newcommand{\acr}{\newline\indent}

\address{\llap{**\,}Department of pure Mathematics\\
Faculty of mathematical
Sciences\\
University of Guilan\\
P. O. Box 41335-19141, Rasht, Iran}
\email{ansari@guilan.ac.ir}

\author{F. Farshadifar*}
\address{\llap{*\,} (Corresponding Author) Assistant Professor, Department of Mathematics, Farhangian University, Tehran, Iran.}
\email{f.farshadifar@cfu.ac.ir}

\author{S. Maleki-Roudposhti***}
\address {\llap{***\,}Department of pure Mathematics, Faculty of mathematical
Sciences, University of Guilan,
P. O. Box 41335-19141, Rasht, Iran.}%
\email{Sepidehmaleki.r@gmail.com}

\begin{abstract}
Let  $R$ be a commutative ring with identity and $M$ be an $R$-module. Let $\psi : S(M)\rightarrow S(M) \cup \{\emptyset \}$  be a function,  where $S(M)$ denote the set of all submodules of $M$.
The main purpose of this paper is to introduce and investigate the notion of strongly $\psi $-2-absorbing second submodules of $M$ as a generalization of strongly 2-absorbing second and $\psi $-second submodules of $M$.
\end{abstract}


\maketitle
\section{Introduction}
\noindent
Throughout this paper, $R$ will denote a commutative ring with
identity and $\Bbb Z$  will denote the ring of integers. We will denote the set of ideals of $R$ by $S(R)$ and the set of all submodules of $M$ by $S(M)$, where $M$ is an $R$-module.

Let $M$ be an $R$-module. A proper submodule $P$ of $M$ is said to be \emph{prime} if for any $r \in R$ and $m \in M$ with $rm \in P$, we have $m \in P$ or $r \in (P:_RM)$ \cite{Da78}. A non-zero submodule $S$ of $M$ is said to be \emph{second} if for each $a \in R$,  the endomorphism of $M$ given by multiplication by $a$ is either surjective
or zero \cite{Y01}.

Let $\phi: S(R)\rightarrow S(R) \cup \{\emptyset \}$ be a function. 
Anderson and Bataineh in \cite{AB08} defined the notation of $\phi$-prime ideals as follows: a proper ideal $P$ of $R$ is \textit{$\phi$-prime} if for r$, s \in R$, $rs \in P \setminus \phi(P)$ implies that $r \in P$ or $s \in  P$ \cite{AB08}.
In \cite{Za10}, the author extended this concept to prime submodule. Let $M$ be an $R$-module. For a function $\phi : S(M)\rightarrow S(M) \cup \{\emptyset \}$, a proper submodule $N$ of $M$ is called \textit{$\phi$-prime} if whenever $r \in R$ and $x \in M$
with $rx \in N \setminus \phi(N)$, then $r \in (N :_R M)$ or $x \in N$.

Let $M$ be an $R$-module and $\psi: S(M) \rightarrow S(M) \cup \{\emptyset \}$ be a function.
Farshadifar and Ansari-Toroghy  in \cite{AF023},  defined the notation of $\psi$-second submodules of $M$ as a dual notion of  $\phi$-prime submodules of $M$. A non-zero submodule
$N$ of $M$ is said to be a \textit{$\psi$-second submodule of $M$} if $r \in R$, $K$ a submodule of $M$, $rN\subseteq K$,  and $r\psi(N) \not \subseteq K $, then $N \subseteq K$ or $rN=0$ \cite{AF023}.

The concept of $2$-absorbing ideals was introduced in \cite{Ba07}.  A proper ideal $I$ of $R$ is said to be a \emph{2-absorbing ideal} of $R$ if whenever $a, b, c \in R$ and $abc \in I$, then $ab \in I$ or
 $ac \in I$ or $bc \in I$. 
  
In \cite{AF16}, the authors introduced the notion of strongly 2-absorbing second submodules as a dual notion of $2$-absorbing submodules and investigated some properties of this class of modules.
A non-zero submodule $N$ of $M$ is said to be a \emph{strongly 2-absorbing second submodule of} $M$ if whenever  $a, b \in R$, $K$ is a submodule of $M$,
and $abN\subseteq K$, then $aN\subseteq K$ or $bN\subseteq K$ or $ab \in Ann_R(N)$
\cite{AF16}.

Let $M$ be an $R$-module and $\psi: S(M) \rightarrow S(M) \cup \{\emptyset \}$ be a function.
The main purpose of this paper is to introduce and investigate the notion of strongly $\psi $-2-absorbing second submodules of $M$ as a generalization of strongly 2-absorbing second and $\psi $-second submodules of $M$.

\section{Main results}
\noindent
\begin{defn}\label{d12.1}
Let $M$ be an $R$-module, $S(M)$ be the set of all
submodules of $M$,   $\psi: S(M) \rightarrow S(M) \cup \{\emptyset \}$ be a function. We say that a non-zero submodule
$N$ of $M$ is a \textit{strongly $\psi $-$2$-absorbing second submodule of $M$} if $a, b \in R$, $K$ a submodule of $M$, $abN\subseteq K$,  and $ab\psi(N) \not \subseteq  K$, then  $a N\subseteq K$ or  $b N\subseteq K$ or  $ab \in Ann_R(N)$.
\end{defn}

In Definition \ref{d12.1}, since $ab\psi(N) \not \subseteq  K$ implies that $ab(\psi(N)+N)\not \subseteq  K$, there is no loss of generality in assuming that $N \subseteq \psi(N)$ in the rest of this paper.

A non-zero submodule $N$ of $M$ is said to be a \emph{weakly strongly 2-absorbing second submodule of} $M$ if whenever
 $a, b \in R$, $K$ is a submodule of $M$, $abM \not \subseteq K$, and $abN\subseteq K$, then $aN\subseteq K$ or
$bN\subseteq K$ or $ab \in Ann_R(N)$ \cite{AFM22}.

Let $M$ be an $R$-module. We use the following functions $\psi: S(M) \rightarrow S(M) \cup \{\emptyset \}$.
$$\psi_i(N)=(N:_MAnn_R^i(N)), \ \forall N \in S(M), \ \forall i \in \Bbb N,$$
$$\psi_\sigma(N)=\sum^{\infty}_{i=1}\psi_i(N), \ \forall N \in S(M).$$
$$\psi_{M}(N)=M, \ \forall N \in S(M),$$
Then it is clear that strongly $\psi_M$-2-absorbing second submodules are weakly strongly 2-absorbing second submodules. Clearly,  for any submodule and every positive integer $n$, we have the
following implications:
$$
strongly \\\ 2-absorbing\\\ second \Rightarrow strongly\\\ \psi_{n-1}-2-absorbing\\\ second
$$
$$
 \Rightarrow strongly\\\ \psi_{n}-2-absorbing\\\ second \Rightarrow strongly\\\  \psi_\sigma-2-absorbing\\\ second.
$$

For functions $\psi, \theta: S(M) \rightarrow S(M) \cup \{\emptyset \}$, we write $\psi \leq \theta $  if $\psi(N) \subseteq \theta (N)$ for each $N \in S(M)$. So whenever $\psi \leq \theta$, any strongly $\psi$-$2$-absorbing second submodule is a strongly $\theta$-$2$-absorbing second submodule.

\begin{rem}
Let $M$ be an $R$-module and $\psi: S(M) \rightarrow S(M) \cup \{\emptyset \}$ be a function. Clearly every strongly 2-absorbing second submodule and every $\psi$-second submodule of $M$ is a strongly $\psi$-$2$-absorbing second submodule of $M$. Also, evidently $M$ is a strongly $\psi_{M}$-2-absorbing second submodule of itself. In particular,   $M=\Bbb Z_6 \oplus \Bbb Z_{10}$ is not strongly 2-absorbing second $\Bbb Z$-module but $M$ is a strongly $\psi_{M}$-2-absorbing second $\Bbb Z$-submodule of $M$.
\end{rem}

In the following theorem, we characterize
strongly $\psi$-2-absorbing second submodules of an $R$-module $M$.
\begin{thm}\label{t2.7}
Let $N$ be a non-zero submodule of an $R$-module $M$ and  $\psi: S(M) \rightarrow S(M) \cup \{\emptyset \}$  be a function.
Then the following are equivalent:
\begin{itemize}
\item [(a)] $N$ is a strongly $\psi$-2-absorbing second submodule of $M$;
\item [(b)] for submodule $K$ of $M$ with $aN \not \subseteq K$ and $a \in R$, we have $(K:_RaN)=Ann_R(aN) \cup (K:_RN)\cup (K:_Ra\psi(N))$;
\item [(c)] for submodule $K$ of $M$ with $aN \not \subseteq K$ and $a \in R$, we have either $(K:_RaN)=Ann_R(aN)$ or $(K:_RaN)=(K:_RN)$ or $(K:_RaN)=(K:_Ra\psi(N))$;
\item [(d)] for each $a, b \in R$ with $ab\psi(N) \not \subseteq abN$, we have either $abN=aN$ or $abN=bN$ or $abN=0$.
\end{itemize}
\end{thm}
\begin{proof}
$(a)\Rightarrow (b)$. Let for a submodule $K$ of $M$ with $aN \not \subseteq K$ and $a \in R$, we have $b \in(K:_RaN)\setminus (K:_Ra\psi(N))$. Then since $N$ is a strongly $\psi$-2-absorbing second submodule of $M$, we have $b \in Ann_R(aN)$ or $bN  \subseteq K$. Thus  $(K:_RaN)\subseteq Ann_R(aN)$ or $(K:_RaN) \subseteq K:_RN)$. Hence, 
$$
(K:_RaN)\subseteq Ann_R(aN) \cup (K:_RN)\cup (K:_Ra\psi(N)).
$$
As we may assume that
$N\subseteq \psi(N) $, the other inclusion always holds.

$(b)\Rightarrow (c)$. This follows from the fact that if an ideal is the union of two ideals, it is equal to one of them.

$(c)\Rightarrow (d)$.
Let $a,b \in R$ such that $ab\psi(N) \not \subseteq abN$ and $aN \not \subseteq abN$. Then by part (c),
we have either $(abN:_RaN)=Ann_R(aN)$ or $(abN:_RaN)=(abN:_RN)$. Hence,  $abN=0$ or $bN\subseteq abN$, as needed.

$(d)\Rightarrow (a)$.
Let $a,b \in R$ and $K$ be a submodule of $M$ such that $abN \subseteq K$ and $ab\psi(N) \not \subseteq K$. If $ab\psi(N) \subseteq abN$, then $abN \subseteq K$ implies that $ab\psi(N) \subseteq K$, a contradiction. Thus by part (d), either $abN=aN$ or $abN=bN$ or $abN=0$. Therefore, $aN \subseteq K$ or $bN \subseteq K$ or $abN=0$ and the proof is completed.
\end{proof}

A proper submodule $N$ of an $R$-module
$M$ is said to be \emph{completely irreducible} if $N=\bigcap _
{i \in I}N_i$, where $ \{ N_i \}_{i \in I}$ is a family of
submodules of $M$, implies that $N=N_i$ for some $i \in I$. It is
easy to see that every submodule of $M$ is an intersection of
completely irreducible submodules of $M$ \cite{FHo06}.

\begin{rem}\label{r2.1} (See \cite{AF101}.)
Let $N$ and $K$ be two submodules of an $R$-module $M$. To prove $N\subseteq K$, it is enough to show that if $L$ is a completely irreducible submodule of $M$ such that $K\subseteq L$, then $N\subseteq L$.
\end{rem}

\begin{thm}\label{t2.3}
Let $M$ be an $R$-module and  $\psi: S(M) \rightarrow S(M) \cup \{\emptyset \}$ be a function.
Let $N$ be a strongly  $\psi$-2-absorbing second submodule of $M$ such that  $Ann_R^2(N) \psi(N)\not\subseteq N $. Then $N$ is a strongly 2-absorbing second submodule submodule of $M$.
\end{thm}
\begin{proof}
Let $a,b \in R$ and $K$ be a submodule of $M$ such that $abN \subseteq K$. If $ab\psi(N) \not \subseteq K$, then we are done because $N$ is a strongly $\psi$-2-absorbing second submodule of $M$. Thus suppose that $ab\psi(N)\subseteq K$. If $ab\psi(N) \not \subseteq N$, then $ab\psi(N) \not \subseteq N \cap K$. Hence $ab N \subseteq N \cap K$ implies that $aN \subseteq N \cap K \subseteq K$ or  $bN \subseteq N \cap K \subseteq K$ or $abN=0$, as needed. So let $ab \psi(N) \subseteq N$. If $aAnn_R(N) \psi(N) \not \subseteq K$, then $a(b+Ann_R(N))\psi(N) \not \subseteq K$. Thus $a(b+Ann_R(N))N \subseteq K$ implies that $aN \subseteq K$ or $bN=(b+Ann_R(N))N \subseteq K$ or $abN=a(b+Ann_R(N))N=0$, as required. So let $aAnn_R(N)\psi(N)\subseteq K$. Similarly, we can assume that $bAnn_R(N)\psi(N) \subseteq K$. Since $Ann_R^2(N)  \psi(N)\not \subseteq N$, there exist $a_1, b_1 \in Ann_R(N)$ such that $a_1b_1 \psi(N)  \not \subseteq N$. Thus there exists a completely irreducible submodule $L$ of $M$ such that $N \subseteq L$ and $a_1b_1\psi(N)  \not \subseteq L$ by Remark \ref{r2.1}. If $ab_1\psi(N) \not\subseteq L$, then $a(b+b_1)\psi(N)  \not \subseteq L \cap K$. Thus $a(b+b_1)N \subseteq L \cap K$ implies that $aN \subseteq  L \cap K \subseteq K$ or $bN=(b+b_1)N \subseteq L \cap K \subseteq K$ or $abN=a(b+b_1)N=0$, as needed. So let $ab_1 \psi(N)  \subseteq L$. Similarly, we can assume that $a_1b \psi(N)  \subseteq L$. Therefore, $(a+a_1)(b+b_1)\psi(N)  \not \subseteq L \cap K$. Hence, $(a+a_1)(b+b_1)N \subseteq L \cap K$ implies that $aN=(a+a_1)N \subseteq K$ or $bN=(b+b_1)N \subseteq K$ or $abN=(a+a_1)(b+b_1)N=0$, as desired.
\end{proof}

Let $M$ be an $R$-module. A submodule $N$ of $M$ is said to be \emph{coidempotent} if  $N=(0:_MAnn_R^2(N))$. Also, $M$ is said to be  \emph{fully coidempotent}  if every submodule of $M$ is coidempotent \cite{AF122}.
\begin{cor}
Let $M$ be an R-module and  $\psi: S(M) \rightarrow S(M) \cup \{\emptyset \}$ be a function.
If $M$ is a fully coidempotent $R$-module and $N$ is a proper submodule of $M$ with $Ann_R(\psi(N))=0$, then $N$ is a strongly  $\psi$-2-absorbing second submodule if and only if $N$ is a strongly 2-absorbing second submodule.
\end{cor}
\begin{proof}
The sufficiency is clear. Conversely, assume on the contrary that $N\not = M$ is a strongly $\psi$-2-absorbing second submodule of $M$ which is not a strongly 2-absorbing second submodule. Then by Theorem \ref{t2.3}, $Ann_R^3(N) \subseteq Ann_R(\psi(N))$. Hence as $Ann_R(\psi(N))=0$, we have $Ann_R^3(N) =0$. Thus since $N$ is coidempotent,
$$
N=(0:_MAnn_R^2(N))=(0:_MAnn_R^3(N))=M,
$$
which is a contradiction.
\end{proof}

\begin{prop}\label{p1.15}
 Let $M$ be an $R$-module and $\psi: S(M) \rightarrow S(M) \cup \{\emptyset \}$ be a function. Let $N$ be a non-zero submodule of $M$. If $N$ is a strongly
$\psi$-$2$-absorbing second submodule of $M$, then for any $a, b \in  R\setminus Ann_R(N)$, we have $abN=aN \cap bN \cap ab\psi(N)$.
\end{prop}
\begin{proof}
Let $N$ be a strongly $\psi$-$2$-absorbing second submodule of $M$ and $ab \in  R\setminus Ann_R(N)$.  Clearly,  $abN\subseteq aN \cap bN \cap ab\psi(N)$.  Now let $L$ be a completely irreducible submodule of $M$ such that $abN \subseteq L$. If $ab\psi(N)  \subseteq L$, then we are done.  If $ab\psi(N) \not \subseteq L$, then $aN \subseteq L$ or $bN \subseteq L$ because $N$ is a  strongly $\psi$-$2$-absorbing second submodule of $M$. Hence $aN \cap bN \cap ab\psi(N)\subseteq L$. Now the result follows from Remark \ref{r2.1}. 
\end{proof}

Let $R_i$ be a commutative ring with identity and $M_i$ be an $R_i$-module for $i = 1,   2$.   Let $R = R_1 \times R_2$.   Then $M = M_1 \times M_2$ is an $R$-module and each submodule of $M$ is in the form of $N = N_1 \times N_2$ for some submodules $N_1$ of $M_1$ and $N_2$ of $M_2$.
\begin{thm}\label{t2.5}
Let $R = R_1 \times R_2$ be a ring and $M = M_1 \times M_2$
be an $R$-module, where $M_1$ is an $R_1$-module and $M_2$ is an $R_2$-module. Suppose that $\psi^i: S(M_i) \rightarrow S(M_i) \cup \{\emptyset \}$  be a function for $i=1, 2$. Then  $N_1 \times 0$ is a strongly $\psi^1 \times \psi^2$-2-absorbing second submodule of $M$,  where $N_1$ is a  strongly $\psi^1$-2-absorbing second submodule of $M_1$ and  $\psi^2(0)=0$.
\end{thm}
\begin{proof}
Let $(a_1, a_2), (b_1, b_2) \in R$ and $K_1 \times K_2$ be a submodule of $M$
such that $(a_1, a_2)(b_1, b_2)(N_1 \times 0) \subseteq K_1 \times K_2$ and
$$
(a_1, a_2)(b_1, b_2) ((\psi^1 \times \psi^2)(N_1 \times 0))=a_1b_1\psi^1(N_1)\times a_2b_2\psi^2(0)
$$
$$
=a_1b_1\psi^1(N_1)\times 0\not\subseteq K_1 \times K_2
$$
Then
 $a_1b_1N_1 \subseteq K_1 $ and $a_1b_1 \psi^1(N_1) \not\subseteq K_1$.
Hence,
$a_1b_1N_1=0$ or $a_1N_1 \subseteq K_1 $ or  $b_1N_1 \subseteq K_1 $ since $N_1$ is a strongly $\psi^1$-2-absorbing second submodule of $M_1$.
Therefore, we have  $(a_1, a_2)(b_1, b_2)(N_1 \times 0)=0 \times 0$ or $(a_1, a_2)N_1 \times 0 \subseteq K_1 \times K_2$ or $(b_1, b_2)N_1 \times 0 \subseteq K_1 \times K_2$, as requested.
 \end{proof}

 \begin{thm}\label{t2.6}
 Let $M$ be an $R$-module and  $\psi: S(M) \rightarrow S(M) \cup \{\emptyset \}$ be a function. Then we have the following.
  \begin{itemize}
    \item [(a)] If $(0:_Mt) \subseteq t \psi((0:_Mt))$,  then $(0:_Mt)$ is a strongly 2-absorbing second submodule if and only if it is a strongly $\psi$-2-absorbing second submodule.
    \item [(b)] If $(tM:_R\psi(tM)) \subseteq Ann_R(tM)$,   then the submodule $tM$ is strongly 2-absorbing second if and only if it is strongly $\psi$-2-absorbing second.
  \end{itemize}
  \end{thm}
  \begin{proof}
  (a) Suppose that $(0:_Mt)$ is a strongly $\psi$-2-absorbing second submodule of $M$,  $a, b \in  R$,    and $K$ is a submodule of $M$ such that $ab(0:_Mt) \subseteq K$. If $ab\psi((0:_Mt)) \not  \subseteq K$,  then since $(0:_Mt)$ is strongly $\psi$-2-absorbing second, we have $a(0:_Mt) \subseteq K$ or $b(0:_Mt) \subseteq K$ or
  $ba \in Ann_R((0:_Mt))$ which implies $(0:_Mt)$ is strongly 2-absorbing second. Therefore we may assume that $ab\psi((0:_Mt)) \subseteq K$.   Clearly,  $a(b+t)(0:_Mt) \subseteq K$.   If $a(b+t)\psi((0:_Mt)) \not \subseteq K$,   then we have $(b+t)(0:_Mt) \subseteq K$ or $a(0:_Mt) \subseteq K$ or $a(b+t)\in Ann_R((0:_Mt))$. Since $at \in Ann_R((0:_Mt))$ therefore $b(0:_Mt) \subseteq K$ or $a(0:_Mt) \subseteq K$ or $ab\in Ann_R((0:_Mt))$. Now suppose that $a(b+t)\psi((0:_Mt)) \subseteq K$. Then since $ab\psi((0:_Mt)) \subseteq K$, we have $ta\psi((0:_Mt)) \subseteq K$ and so $t\psi((0:_Mt)) \subseteq (K:_Ma)$. Now $(0:_Mt) \subseteq t\psi((0:_Mt))$ implies that  $(0:_Mt) \subseteq (K:_Ma)$. Thus $a(0:_Mt) \subseteq K$, as needed. The converse is clear.

  (b) Let $tM$ be a  strongly $\psi$-2-absorbing second submodule of $M$ and assume that $a,b \in R$ and $K$ be a submodule of $M$ with
  $abtM \subseteq K$.  Since $tM$ is  strongly $\psi$-2-absorbing second submodule, we can suppose that $ab\psi(tM) \subseteq K$,   otherwise $tM$ is strongly 2-absorbing second. Now $abtM \subseteq tM \cap K$.  If $ab\psi(tM) \not \subseteq tM \cap K$,  then as $tM$ is  strongly $\psi$-2-absorbing second submodule, we are done. So let $ab\psi(tM) \subseteq tM \cap K$. Then $ab\psi(tM) \subseteq tM$. Thus $(tM:_R\psi(tM)) \subseteq Ann_R(tM)$ implies that $ab \in Ann_R(tM)$, as requested. The converse is clear.
\end{proof}

\end{document}